\documentclass[12pt]{amsart}
\usepackage{latexsym,fancyhdr,amssymb,color,amsmath,amsthm,graphicx,listings,comment}
\newtheorem{conj}{Question} 
\newtheorem{thm}{Theorem}       \newtheorem{propo}{Proposition}
\newtheorem{lemma}{Lemma}       \newtheorem{coro}{Corollary}
\usepackage[section]{placeins}      \let\paragraph\subsection
\def\B#1#2{{#1\choose #2}}

\title{The average simplex cardinality of a finite abstract simplicial complex}
\author{Oliver Knill} \date{5/5/2019}
\address{Department of Mathematics \\ Harvard University \\ Cambridge, MA, 02138 }

\begin{document}

\begin{abstract}
We study the average simplex cardinality 
${\rm Dim}^+(G) = \sum_{x \in G} |x|/(|G|+1)$
of a finite abstract simplicial complex $G$. It also defines
an average dimension ${\rm Dim}(G)={\rm Dim}^+(G)-1$. 
The functional ${\rm Dim}^+$ is a homomorphism from the monoid of simplicial complexes 
$\mathcal{G}$ to $\mathbb{Q}$: the formula
${\rm Dim}^+(G \oplus H) = {\rm Dim}^+(G) + {\rm Dim}^+(H)$ holds for the join 
$\oplus$ similarly as for the augmented inductive dimension 
${\rm dim^+}(G) = {\rm dim}(G)+1$ where ${\rm dim}$ is the inductive dimension
(as recently shown by Betre and Salinger \cite{BetreSalinger}). 
In terms of the generating function $f(t) = 1+v_0 t + v_1 t^2 + \cdots +v_d t^{d+1}$
defined by the f-vector $(v_0,v_1, \dots, v_d)$ of $G$ for which $f(-1)$ is 
genus $1-\chi(G)$ and $f(1)$ $={\rm card}(G) = |G|+1$ is the augmented number 
of simplices, the average cardinality is the logarithmic derivative 
${\rm Dim}^+(f) =  f'(1)/f(1)$ of $f$ at $1$.
So, $e^{{\rm dim^+}(G)}$, $e^{{\rm Dim}^+(G)}$, 
$e^{\rm max(G)}$, ${\rm genus}(G)=1-\chi(G)=f(-1)$, 
or ${\rm card}^+(G)=f(1)$ are all multiplicative characters from 
$\mathcal{G}$ to $\mathbb{R}$.  
After introducing the average cardinality and establishing its compatibility 
with arithmetic, we prove two results: 1) the inequality 
${\rm dim}^+(G)/2 \leq {\rm Dim}^+(G)$ with equality for complete complexes. 
2) In the Barycentric limit, $C_d=\lim_{n \to \infty} {\rm Dim}^+(G_n)$ 
is the same for any initial complex $G_0$ of maximal dimension $d$ and
the constant $C_d$ is explicitly given in terms of the Perron-Frobenius eigenfunction
of the universal Barycentric refinement operator $f_{G_{n+1}} = A f_{G_n}$ and is for 
positive $d$ always a rational number in the open interval $((d+1)/2,d+1)$. 
\end{abstract}

\maketitle

\section{Dimensions}

\paragraph{}
A finite set $G$ of non-empty sets closed under the operation of taking finite non-empty
subsets is called an {\bf abstract finite simplicial complex}. The sets $x$ in $G$ are also called 
{\bf simplices} or (if the complex is the Whitney complex of a graph, the sets are also called 
{\bf cliques} and the Whitney complex is known as clique complex). 
The {\bf dimension} of a set $x \in G$ is $|x|-1$, one less than the cardinality 
$|x|$ of the simplex $x$. 

\paragraph{}
One sometimes also looks at the {\bf augmented complex}
$G^+ = G \cup \{\emptyset\}$ which is a set of sets closed under the operation of taking arbitrary subsets. 
The {\bf maximal cardinality} ${\rm max}(G)$ is by one larger than the {\bf maximal
dimension} of the complex. A finite abstract simplicial complex $G$ has the {\bf maximal dimension}
${\rm max}_{x \in G} {\rm dim}(x)$, and the {\bf inductive dimension}
$1+(1/|G|) \sum_{x \in G} {\rm dim}(S(x))$, where
$S(x)=\{ y \in G \; | \;  y \subset x$ or $x \subset y\}$ is the {\bf unit sphere} of $x \in G$.
These dimensions are defined for arbitrary sets of sets not only for simplicial complexes. 

\paragraph{}
The inductive dimension of $G$ is a rational number less or equal than the maximal dimension of $G$. 
It behaves nicely with respect to the Cartesian product because 
of the inequality ${\rm dim}(G \times H) \geq {\rm dim}(G) + {\rm dim}(H)$
for the Cartesian product \cite{KnillKuenneth} (the Barycentric refinement of the 
Cartesian set product $G \times H$), 
which holds similarly also for the Hausdorff dimension does \cite{Falconer} (formula 7.2) 
for Borel sets in Euclidean space. 

\paragraph{}
Recently, in \cite{BetreSalinger} it was proven for the inductive dimension ${\rm dim}$ that the functional
${\rm dim}^+(G) = {\rm dim}(G) + 1$ is additive:
$$  {\rm dim}^+(G \oplus H) = {\rm dim}^+(G) + {\rm dim}^+(H) \; . $$ 
This is a bit surprising as it is an identity, where one has rational numbers on both sides.

\begin{figure}[!htpb]
\scalebox{0.22}{\includegraphics{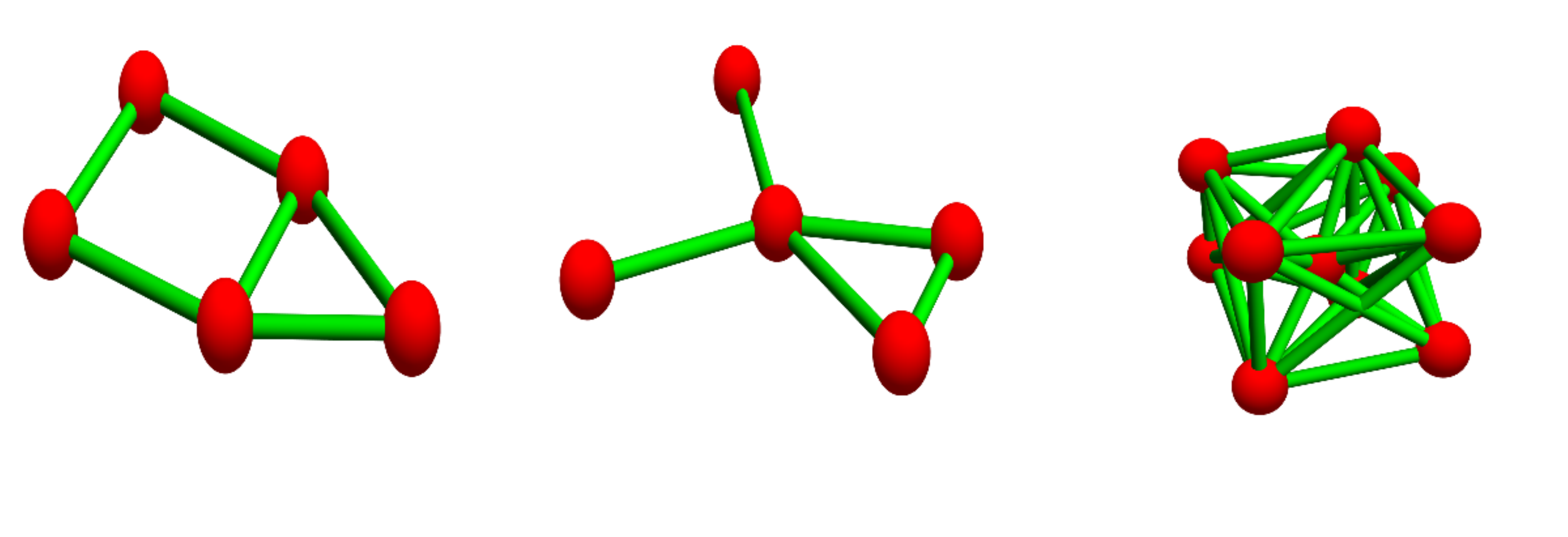}}
\label{betresalinger}
\caption{
An illustration of the compatibility equality telling that the augmented dimension
of $G \oplus H$ is the sum of the augmented dimensions of  $G$ and $H$. 
The first graph is the ``house graph", the second the ``rabbit graph".
The average simplex cardinalities are ${\rm Dim}^+(G) = 20/13$, 
${\rm Dim}^+(H)=3/2$, ${\rm Dim}^+(G \oplus H)=79/26=20/13+3/2$. The
augmented inductive dimensions are 
${\rm dim}^+(G)=37/15$, ${\rm dim}^+(H)=5/2$ and 
${\rm dim}^+(G \oplus H)= 149/30=37/15+5/2$. The augmented maximal dimensions
are ${\rm max}(G) = 3$, ${\rm max}(H) = 3$ and ${\rm max}(G \oplus H) = 6=3+3$. 
}
\end{figure}

\paragraph{}
We introduce and study here the {\bf average dimension} 
$$   {\rm Dim}(G)=\frac{1}{|G|+1} \sum_{x \in G} {\rm dim}(x) $$
and its augmentation 
$$   {\rm Dim}^+(G)=\frac{1}{|G|+1} \sum_{x \in G} |x| \; , $$   
which we call the {\bf average simplex cardinality} of simply the {\bf average cardinality} 
of the simplicial complex $G$.
It measures the expected size $|x|$ of a set $x \in G$. For a triangle 
$K_3^+=\{ \{1,2,3\},\{1,2\},\{1,3\},\{2,3\},\{1\},\{2\},\{3\},\emptyset \}$ 
for example, the average simplex cardinality is 
$(1 \cdot 3 + 3 \cdot 2 + 3 \cdot 1+1 \cdot 0)/8 = 12/8=3/2$. 

\paragraph{}
The use of $|G|+1$ instead of $|G|$ in the denominator makes the arithmetic nicer. 
One can justify it also by looking at the {\bf augmented complex} $G^+$ which includes the empty set $\emptyset$ 
in $G$ and assign it to have dimension $|\emptyset|-1 = (-1)$ which has cardinality or augmented dimension $0$ 
so that it has no weight when taking the expectation but which adds $1$ to the cardinality of the sets in $G$.
In other words $|G|+1=|G^+|$ can explain the augmentation.

\paragraph{}
The presence of the empty set $\emptyset$ is also reflected in the use of the {\bf simplex generating function}
$$ f_G (t) = 1+v_0 t + \cdots + v_d t^{d+1} \; , $$
where $v_k$ counts the number of $k$-dimensional simplices in $G$. 
The constant term in $f$ can then be interpreted as $1 v_{-1}$, where $v_{-1}=1$ counts the number of
$(-1)$-dimensional sets (the empty set $\emptyset$). We prefer to stick here however to the more traditional definition in 
topology and call ``simplices" the non-empty 
sets in $G$ and see the {\bf empty set} as a $(-1)$ dimensional {\bf simplicial complex with no elements} and 
{\bf not as a $(-1)$-dimensional set in a simplicial complex}.

\paragraph{}
The {\bf join} of two simplicial complexes $G,H$ is defined as the complex
$$    G \oplus H = G \cup H \cup \{ x=y \cup z \; | y \in G, z \in H  \} \; . $$
Its graph is the {\bf Zykov join} of the graphs $\Gamma(G)$ and $\Gamma(H)$ \cite{Zykov},
which is the graph obtained by taking the disjoint union of the vertex sets of $G$ and $H$ and connecting 
any vertex in $G$ with any vertex in $H$. The Zykov join is defined for all graphs, they don't need to
come from a simplicial complex. 

\paragraph{}
The {\bf join monoid} $(\mathcal{G},\oplus)$ has the class of spheres as a sub-monoid for 
which the augmented dimension ${\rm Dim}^+(G$) is additive; this is all completely equivalent to 
standard topology, where the same identities hold. Also as in the continuum, 
the join with a $1$-point graph is called a {\bf cone extension} and
the join with the $2$-point graph (the 0-sphere) is the {\bf suspension}. 

\paragraph{}
From this point of view, also the {\bf genus} $f(-1) = 1-\chi(G)$ 
(which is sometimes called the {\bf augmented Euler characteristic}) is a valid 
alternative to the Euler characteristic $\chi(G)$ as it is compatible with the arithmetic: 
it multiplies with the join operation. 

\paragraph{}
Still, as Euler characteristic and simplicial complexes without empty set are 
far more entrenched outside of combinatorial topology, we prefer to stick to{\bf simplicial 
complex} and {\bf Euler characteristic}, as custom in topology. 
We will see in a moment that the average cardinality is the logarithmic derivative of $f$ at $1$. This immediately 
implies compatibility of ${\rm Dim}^+(G)$ with the join operation $\oplus$, facts which are trivial 
for the maximal dimension ${\rm max}(G)$ and which was recently shown to hold for the augmented inductive dimension 
${\rm dim}^+$ by Betre and Salinger \cite{BetreSalinger}.

\paragraph{}
The average dimension ${\rm Dim}$ is most likely of {\bf less topological importance} than 
the inductive dimension but instead of {\bf more statistical interest}. 
Both the maximal dimension ${\rm max}$ as well as the inductive dimension ${\rm dim}$ 
are more intuitive than the average dimension ${\rm Dim}$, as they align with dimension notions we are 
familiar with in the continuum. Inductive dimension is an integer for discrete manifolds for example. 
The average dimension however gives a richer statistical description of a network as tells about the
distribution of the various simplices. 

\paragraph{}
The fact that discrete manifolds have integer dimensions is a bit more general even: 
a simplicial complex $G$ can inductively be declared to
be a {\bf discrete $d$-variety} if every unit sphere $S(x)$ is a discrete $(d-1)$-variety. 
The class {\bf $d$-varieties} is more basic than $d$-manifolds as no homotopy is involved 
in its definition. For {\bf discrete d-manifolds}, we insist
that the unit spheres $S(x)$ are all $(d-1)$-spheres, which are defined to be $(d-1)$-manifolds 
which when punctured become contractible. The figure $8$-graph is an example of a $1$-variety which is not 
a discrete $1$-manifold, because one of its vertices has a unit sphere which is not a $0$-sphere.

\paragraph{}
The inductive dimension of ``discrete variety" is booted up with the assumption that the {\bf empty
complex} $0$ is a discrete variety. It is obvious from the definition that both maximal dimension
as well as inductive dimensions of a variety are integers and that they agree with the maximal dimension. 
But this is not true for the average cardinality in general. 
The octahedron graph and  icosahedron graph both have inductive dimension $2$.
The average cardinality of the octahedron graph is $2$, the average cardinality of the icosahedron
graph is $44/21=2.095..$. The Barycentric refinement of the octahedron graph has
average cardinality $314/147=2.136...$ while the Barycentric refinement of the icosahedron graph 
has average cardinality $782/363=2.15...$. In dimension 2, the limiting average cardinality is
$(1,3,2) \cdot (1,2,3)/(1+3+2)=16/6=2.1666 ...$ as $(1,3,2)$ is the dominant eigenvector of the 
Barycentric operator 
$A=\left[ \begin{array}{ccc} 1 & 1 & 1 \\ 0 & 2 & 6 \\ 0 & 0 & 6 \\ \end{array} \right]$ in 
dimension $2$. 

\paragraph{}
{\bf Examples:} for the complete graph $K_n$, the cyclic graph $C_n$ and the path graphs $P_n$ 
and the complete bipartite graphs $K_{n,n}$ or the point graphs $E_n=1+1 \cdots + 1$ with $n$ vertices 
and no edges: \\
${\rm Dim}^+(E_n)=n/(n+1)$         \hfill ${\rm dim}^+(E_n)=1$. \\
${\rm Dim}^+(K_n)=n/2$             \hfill ${\rm dim}^+(K_n)=n$. \\
${\rm Dim}^+(C_n)=3n/(2n+1)$       \hfill ${\rm dim}^+(C_n)=2$. \\
${\rm Dim}^+(P_n)=(3n-2)/(2n-1)$   \hfill ${\rm dim}^+(P_n)=2$. \\
${\rm Dim}^+(K_{3,3}) = 3/2$       \hfill ${\rm dim}^+(K_{3,3})=2$. \\
${\rm Dim}^+(K_{n,n}) = 2-2/(1+n)$ \hfill ${\rm dim}^+(K_{n,n})=2$. \\

\paragraph{}
For a one-dimensional simplicial complex with Euler characteristic $\chi=|V|-|E|=n-m$, we have
${\rm Dim}^+(G)=(|V|+2 |E|)/(|V|+|E|+1) = (3n-2\chi)/(2n-\chi+1)$. We see from this
in particular that in the Barycentric limit of a one-dimensional complex we get the average simplex cardinality 
value $3/2$ in dimension $d=1$. This can be expressed with the dominant eigenvector of the 
Barycentric refinement operator, the eigenvector to the maximal eigenvalue $d!$. 
The algebraically confirmation will be done in full generality; but in this particular case of 
dimension $d=1$, the Barycentric refinement operator
$A=\left[ \begin{array}{cc} 1 & 1 \\ 0 & 2 \\ \end{array} \right]$ has the 
{\bf Perron-Frobenius eigenvector}
$(1,1)$ and $(1,1)\cdots (1,2)/(1+1)=3/2$. 

\paragraph{}
Remarkable for the utility graph $K_{3,3}$ is that {\bf its average simplex dimension does not change
under Barycentric refinements}. It remains $3/2$. This is why we can not claim ${\rm Dim}^+(G_1) > {\rm Dim}^+(G)$ in 
general. It also is not the case for the zero-dimensional case $n=1 + 1 + \cdots + 1$ but there are no
two or higher dimensional simplicial complexes $G$ for which ${\rm Dim}^+(G_1) = {\rm Dim}^+(G)$. 
In dimension $d=2$, where the limiting Barycentric limit dimension is $23/8=2.875$, we would 
need a complex of which the vertex,
edge and triangle cardinalities $v_0,v_1,v_2$ would have to satisfy $(v_0+2 v_1 + 3 v_2)/(1+v_0+v+1+v_2)=23/8$
which implies $v_2 = 23 + 15 v_0 + 7 v_1$ which is not possible as $v_2 \leq v_0$.
In one-dimension we just need a graph with $v_1 = 3 +v_0$ to assure $(v_0+2 v_1)/(1+v_0+v_1) = 3/2$. There are many 
one-dimensional simplicial complexes beside the utility graph for which this happens. 

\section{Dimension expectation} 

\paragraph{}
If $G$ is a finite abstract simplicial complex of maximal dimension $d$, its {\bf $f$-vector of $G$} is
$(v_0,v_1,\cdots, v_d)$, where $v_k$ is the number of $k$-dimensional simplices in $G$. 
If $f(t)=1+v_0 t + v_1 t^2 + \cdots + v_d t^{d+1}$ is the {\bf simplex generating function}, then 
the {\bf average simplex cardinality} is defined as 
$$  {\rm Dim}^+(G)=\sum_{x \in G} \frac{{\rm dim}(x)+1}{|G|+1} = \frac{f'(1)}{f(1)} \; . $$

\paragraph{}
We are also interested in the logarithm $g(t) = \log(f(t))$ because
$g_G(-1) = \log(1-\chi(G))$ is the logarithm of the genus which is additive
and because $g_G'(1) = {\rm Dim}^+(G)$ is the average simplex cardinality.
The rational function 
$$  g'(t)= \frac{f'(t)}{f(t)}   $$ 
is therefore of interest: 

\begin{lemma}
$g'=f'/f: \mathcal{G} \to \mathbb{Q}(t)$ is an additive homomorphism:
$$ g_{G \oplus H} (t) = g_G(t) + g_H(t) $$
For any $a>0$, the functional $G \to f_G(a)/f'_G(a)$ is an additive 
homomorphism from the join-monoid of simplicial complexes to $\mathbb{Q}$. 
\end{lemma}
\begin{proof}
If follows immediately from $f_{G \oplus H} = f_G f_H$ which implies that
any derivative of $\log(f_G)$ is additive. \\
The second statement follows from the fact that if $f$ is a polynomial with 
positive entries, then all roots of $f$ and $f'$ are negative. 
\end{proof}

\paragraph{}
It follows that we can relate the value of $g$ on a {\bf unit ball} $B(x)$
with the value of the {\bf unit sphere} $S(x)$ as $B(x) = S(x) \oplus K_1$. 

\begin{coro}[Unit sphere and Unit ball dimension]
$g_{B(x)}(t) = g_{S(x)}(t) + 1/(1+t)$. Especially 
$$  {\rm Dim}^+(B(x)) = {\rm Dim}^+(S(x)) + 1/2 \; . $$
\end{coro}
\begin{proof}
$f_{B(x)}(t) = f_{S(x)}(t) (1+t)$
and $g_{B(x)}(1) = g_{S(x)}(1) + 1/2$. 
\end{proof}

\paragraph{}
We can evaluate the rational function at some points to get additive homomorphisms.
For example $G \to f_G'(-1)/f_G(-1)$ is an additive homomorphism. It is sort of a {\bf super
cardinality average} as the degree $f(-1)=1-\chi(G)$ is sort of a super cardinality of the 
complex and $f'(-1)$ counts up the simplices with sign. But $f'(-1)/f(-1)$ is
only defined for complexes $G$ which do not have Euler characteristic $\chi(G)=1$. 
Evaluating $f'(t)/f(t)$ for positive $t$ however is never a problem as all roots of $f$
are negative (as this holds for any polynomial with non-negative entries).  

\section{Examples}

\paragraph{}

\begin{propo}
$$  {\rm Dim}^+(K_n) = \frac{n}{2} = \frac{{\rm dim}^+(K_n)}{2} \; .  $$
If the equality ${\rm Dim}^+(G) = {\rm dim}^+(G)/2$ holds, then $G$ is $K_n$. 
\end{propo}

\begin{proof}
For complete graphs $K_n$, we have $f'(t)/f(t) = n/(1+t)$ because $f_G(x) = (1+t)^n$
as $K_n = K_1 \oplus K_1 \cdots \oplus K_1$ is the join of $n$ one-point graphs $K_1$. 
We therefore have $f'(1)/f(1)=n/2$. The inequality follows from 
${\rm dim}^+(K_n)=n$.
\end{proof} 

\paragraph{}
The {\bf cross polytopes} $S^d$ are $d$-spheres which can be written as an iterated suspension 
$2 \oplus 2 \cdots \oplus 2$, where $G=2$ is the $0$-sphere, which is the zero-dimensional complex
with $2$ vertices. For $d=1$, the cross-polytop is $C_4$ and for $d=2$, it is the octahedron. 
For the $0$-sphere $2$, we have $f(t)=1+2t$ so that $f'(1)/f(1)=2/3$. 
Therefore, the average dimension is $f'(t)/f(t) = (d+1) 2/(1+2t)$. 
This gives for the $0$-sphere $2/3$. For $1$-sphere $4/3$. \\

\paragraph{}
It follows quite directly from the {\bf Dehn-Sommerville relations} that 
for any even dimensional sphere the function $f$ always has a root at $x=-1/2$
while the other roots pair up to $-1$ (we will explore this a bit more elsewhere).
For every odd-dimensional sphere in particular, the poles of $f$ pair up to $-1$. 
For us, the Dehn-Sommerville topic is both a subject of excitement as well as 
of disappointments as we have met manifestations of it being under the impression to 
``discover" it: a first case was the fact that the eigenvectors of the transpose $A^T$ 
of the Barycentric refinement operator $A_d$ define functionals which lead to quantities which 
must be zero for manifolds \cite{DehnSommerville}. The topic however comes
Also the occurrence of $-1/2$ as roots of the simplex generating 
functions of even dimensional spheres first led to excitement until realizing that it is an
other manifestation of Dehn-Sommerville. 

\paragraph{}
To see a factor $1+2x$ in the simplex generating function for even-dimensional spheres
is not so obvious because the roots in general change with homotopy deformations. 
The root at $x=-1/2$ also appears for other discrete manifolds with Euler characteristic $2$ 
like the disjoint union $G$ of two projective planes. The root $-1/2$ happens also for 
homology spheres (a complex which is not a sphere but which has the same homology and especially the 
same cohomology than the standard $3$-sphere). There are homology spheres however where the 
roots can become complex. And of course, because $f$ has positive entries, all roots of all 
derivatives of $f$ are negative and in the sphere case are contained in the open interval $(-1,0)$. 

\section{Lower bound from inductive dimension.}

\paragraph{}
We now compare the two additive functionals ${\rm dim}^+$ and ${\rm Dim}^+$. 
We consider it the main result of this paper as it is not totally obvious. 

\begin{thm}
For all simplicial complexes $G$, the inequality 
$$  \frac{{\rm dim}^+(G)}{2} \leq {\rm Dim}^+(G) \; $$
holds. 
\label{dimensioninquality}
\end{thm}

\paragraph{}
{\bf Examples.} \\
{\bf 1)} For complexes with maximal dimension $0$, the number $n$ of $0$-dimensional sets
determines the complex. The left hand side ${\rm dim}^+(G)/2$ is then always $1/2$. 
The right hand side is $n/(n+1)$. They are equal for $n=1$ only.\\
{\bf 2)} For complexes with maximal dimension $1$, there are two numbers $n,m$ which determine
the complex and ${\rm dim}^+(G)/2$ is in the interval $(1/2,1]$. Now, because $m \geq 1$
(we would with $m=0$ have a zero dimensional complex), we have
${\rm Dim}^+(G)  = (n+2m)/(n+m+1) \geq 1$. There is only one possibility that $(n+2m)/(n+m+1)=1$
and that is $m=1$. If $m=1$ and $n$ is larger than $2$, then ${\rm dim}^+(G)/2<1$. So, the
inequality is strict if $G$ is not the complete one-dimensional complex. \\
{\bf 3)} For circular graphs $C_n$ in particular we have ${\rm dim}^+(C_n)=2$ and ${\rm Dim}^+(C_n)=n/(n+1)$.

\paragraph{}
If $f$ is a functional on simplicial complexes, 
denote by ${\rm E}_G[f]$ the {\bf expectation} of $f$ when evaluated on 
all unit spheres $S(x)$ of $G$: 
$$  {\rm E}_G[f] =  \sum_{x \in G} f(S(x))/|G| \; . $$

\begin{lemma}
a) ${\rm Dim}^+(G)  \geq  \frac{1}{2} + {\rm E}_G[{\rm Dim}^+]$. \\
b) If equality holds in a), then $G=K_n$. 
\end{lemma}
\begin{proof}
a) Every simplex $X$ in $S(x)$ corresponds to a simplex $X \cup x$ in $G$ and 
its average count is $1/2$ larger. This is reflected in the formula 
${\rm Dim}^+(B(x)) = 1/2 + {\rm Dim}^+(S(x))$.
Now we can use an elementary averaging (conditional expectation) result
for finite probability spaces $\Omega$: if $A_n$ is a covering of $\Omega$ and $X$ is a
non-negative random variable, then ${\rm E}_G[ E[X|A_j] ] \geq {\rm E}_G[X]$. This becomes
even strict if we average by the cardinality +1 rather than the cardinality. 
We can apply this now to the situation where the unit balls $B(x)$ cover $G$. \\
b) Equality holds in the later case if and only if $j \to  {\rm E}_G[X|A_j]$ is constant
which means that $G$ is a complete complex. 
If $A_j=S(x_j)$ is not a complete graph at some point, we have an inequality. 
\end{proof}

\paragraph{}
Let's now turn to the proof of the theorem~(\ref{dimensioninquality}) which 
states that  ${\rm dim}^+(G)/2 \leq {\rm Dim}^+(G)$.
\begin{proof}
We prove this by induction with respect to the maximal dimension $d$.
Assume things work in maximal dimension $(d-1)$.
This means ${\rm dim}^+(S(x))/2 \leq {\rm Dim}^+(S(x))$ 
or 
$$ {\rm dim}^+(S(x)) \leq 2 {\rm Dim}^+(S(x))  \; . $$
Use the inductive definition of dimension on the left
and the definition of expectation on the right, we get
$$ {\rm dim}^+(G) -1 \leq 2 {\rm E}_G[{\rm Dim}^+]  \; . $$
From the lemma, which tells that 
$$ {\rm E}_G[{\rm Dim}^+] \leq  {\rm Dim}^+(G)  - \frac{1}{2} $$
we conclude 
$$ {\rm dim}^+(G) -1 \leq 2 {\rm Dim}^+(G) -1 \; . $$
After adding $1$ on both sides and dividing by $2$ we end up
with the claim of the theorem. 
\end{proof} 

\paragraph{}
The inequality is sharp for complete complexes $K_n$. It is strict for
discrete unions of complete complexes and gets even bigger if complete complexes are 
joined along lower dimensional complexes.

\paragraph{}
We also tried first to see what happens if we look at the complex as a CW complex
and add a cell. Both the inductive dimension and average dimension are still
defined then. What happens is that we change the $f$-vector at one point but it 
is not clear what happens with $f'(1)/f(1)$ after such a choice. There is no
question however that the inequality holds for structures which are more general
than simplicial complexes. Discrete CW complexes are examples. 

\section{Arithmetic compatibility}

\paragraph{}
A recent result of Betre and Salinger \cite{BetreSalinger} tells that
$$ {\rm dim}^+(G \oplus H) = {\rm dim}^+(G) + {\rm dim}^+(H) \; . $$
This result for inductive dimension is not so obvious. It is much easier 
to prove the same formula for the average simplex cardinality. There is no
relation however and the following result does not imply the Betre-Salinger
equality: 

\begin{propo} 
${\rm Dim}^+(G \oplus H) = {\rm Dim}^+(G) + {\rm Dim}^+(H)$. 
\end{propo}
\begin{proof}
As ${\rm Dim}^+(G) = f_G'(1)/f_G(1)$, this immediately follows from the Leibniz product rule  
$(fg)'/(fg)=f'/f+g'/g$, which holds for all $t$. 
\end{proof} 

\paragraph{}
This formula brings in arithmetic compatibility similarly as the {\bf genus functional}
$$ {\rm gen}(G) = 1-\chi(G)  \; , $$
where $\chi$ is the Euler characteristic. The identity
$$   {\rm gen}(G \oplus H) = {\rm gen}(G) {\rm gen}(H) $$ 
was useful for example when looking at the hyperbolic decomposition of a unit sphere with respect to the
counting function $|x|$ on $G$ which is Morse: we have $S(x) = S^-(x) \cup S^+(x)$,
where $S^-(x)$ is the sphere consisting of all simplices contained in $x$ and
$S^+(x)$ is the {\bf star} of $x$, the set of simplices containing $x$. Unlike $S^-(x)$
which is a simplicial complex, the star is in general not a simplicial complex, 
but its Barycentric refinement is.

\paragraph{}
Because ${\rm gen}(S^-(x)) = \omega(x)=(-1)^{|x|-1}$ we got the formula
${\rm gen}(S(x)) = 1-\chi(S(x)) = \omega(x) {\rm gen}(S^+(x))$. Because ${\rm gen}(S^+(x))$ is a
{\bf Poincar\'e-Hopf index} of the function $g(x)=-|x|$, we had
$\sum_x \omega(x) (1-\chi(S(x))) = \chi(G)$ which is a McKean-Singer type formula, the reason
for the terminology being that the super trace of $g=L^{-1}$ is $\chi(G)$, where $L$ is the
connection Laplacian of $G$, the reason being that the diagonal entries of $g(x,x)$ is the genus
$1-\chi(S(x))$ of the unit sphere $S(x)$. The Euler characteristic is by definition also the super trace
$\sum_x \omega(x) L(x,x)$ of $L$. The McKean-formula for the Hodge Laplacian is
${\rm str}(e^{-t H}) = \chi(G)$ which holds not only in differential geometry but works for
simplicial complexes.

\paragraph{}
A consequence of the Betre-Salinger identity is that one can naturally extend
the inductive dimension ${\rm dim}^+$ to the Zykov group defined by doing the 
{\bf Grothendieck completion} of the monoid $\oplus$. After that completion, there are complexes of 
positive sign as well as of negative sign, completely analogue as the negative numbers were introduced
in arithmetic. By extending the functional additively, we see that a negative complex $-G$ of a 
standard simplicial complex has a negative dimension. We summarize:

\begin{coro}
Inductive, average and maximal dimensions all extend to an additive group homomorphism
from the Zykov group of simplicial complexes and gives a result in the additive group $\mathbb{Q}$. 
The zero element, the empty complex $0$, gets mapped to $0$.
\end{coro}

\paragraph{}
Extending the genus $1-\chi(G)$ to the Zykov ring looks trickier at first because
$-G$ is $1/{\rm gen}(G)$ which is infinite if the genus of $G$ is zero.
The simple remedy is to look at the generating function
$f_G(t) = 1+v_0 t + v_1 t^2 + \cdots$ with $f$-vector $(v_0,v_1, \dots)$ of $G$.
which is multiplicative
$$  f_{G \oplus H}(t) = f_G(t) f_H(t) \; . $$
It is therefore possible to extend $f_G$ to the Zykov group, where it becomes a rational function.
Extending the rational function to the Zykov group is not a problem, extending particular values
can be, as the values can become infinity at poles of the rational function. 

\section{Barycentric refinement}

\paragraph{}
The {\bf Barycentric refinement} $G_1$ of $G$ is a new complex which 
can be seen as the set of vertex sets of complete subgraphs of the graph
with vertex set $G$ and edge set $(x,y)$ for which either $x \subset y$ or $y \subset x$.
One can also define it without graphs as the {\bf order complex} of $G$. The sets of $G_1$
are all subsets $A$ of $2^G$ which have the property that all elements of $A$
are pairwise contained in each other. 

\paragraph{}
The Barycentric refinement process has some nice limiting properties. (See 
\cite{AmazingWorld} for an overview.) An example is that the spectra of the 
Kirchhoff Laplacian $L(G)$ have a {\bf density of states} which converge to a limiting
measure which only depends on the maximal dimension of the complex \cite{KnillBarycentric2}.
An other example is that the zeta function of the connection Laplacian converges in the one-dimensional
case to a nice limiting zeta function. Unlike for the Kirchhoff Laplacians, where the 
limiting zeta function is difficult \cite{KnillZeta}, the connection Laplacian case
is easier \cite{DyadicRiemann} for zeta and because the connection Laplacians are unimodular.

\paragraph{}
Depending on the maximal dimension $d$ there is a $(d+1) \times (d+1)$ matrix 
$$  A_{ij}  = {\rm Stirling}(j,i) i!   \;  $$
such that $f_{G_1} = A f_G$. While this induces a linear map on $f_G'(1)$, it only induces
an affine map on $f_G(1)$. It also induces a transformation on the probability vector $f/|f|_1$.
Here is the matrix $A$ in the case $d=10$: 
$$ A_{10} = \left[
                 \begin{array}{ccccccccccc}
                  1 & 1 & 1 & 1 & 1 & 1 & 1 & 1 & 1 & 1 & 1 \\
                  0 & 2 & 6 & 14 & 30 & 62 & 126 & 254 & 510 & 1022 & 2046 \\
                  0 & 0 & 6 & 36 & 150 & 540 & 1806 & 5796 & 18150 & 55980 & 171006 \\
                  0 & 0 & 0 & 24 & 240 & 1560 & 8400 & 40824 & 186480 & 818520 & 3498000 \\
                  0 & 0 & 0 & 0 & 120 & 1800 & 16800 & 126000 & 834120 & 5103000 & 29607600 \\
                  0 & 0 & 0 & 0 & 0 & 720 & 15120 & 191520 & 1905120 & 16435440 & 129230640 \\
                  0 & 0 & 0 & 0 & 0 & 0 & 5040 & 141120 & 2328480 & 29635200 & 322494480 \\
                  0 & 0 & 0 & 0 & 0 & 0 & 0 & 40320 & 1451520 & 30240000 & 479001600 \\
                  0 & 0 & 0 & 0 & 0 & 0 & 0 & 0 & 362880 & 16329600 & 419126400 \\
                  0 & 0 & 0 & 0 & 0 & 0 & 0 & 0 & 0 & 3628800 & 199584000 \\
                  0 & 0 & 0 & 0 & 0 & 0 & 0 & 0 & 0 & 0 & 39916800 \\
                 \end{array} \right]  \; . $$

\paragraph{}
The eigenvalues are the diagonal entries $\lambda_k=k!$. The eigenfunctions $v_k$ are the 
{\bf Perron-Frobenius eigenvectors} of $A_k$. They all have non-negative components. 
The matrix is only non-negative but has the property
that given a vector with positive entries (as must happen for $f$-vectors of simplicial complexes), 
the map $\overline{f} \to \overline{A} \overline{f} = \overline{Af}$ on directions $\overline{f}/|f|$ 
is a contraction. There exists therefore a unique fixed point of $\overline{A}$. This is the Perron-Frobenius
eigenvector of $A$ and is the stable probability distribution $f/|f|_1$ in dimension $d$. 
Since the limiting mean of this probability distribution is the limiting average simplex cardinality, 
we have

\begin{thm}
Given a simplicial complex $G$ of maximal dimension $d$. The successive Barycentric refinements $G_n$ 
have a unique limit $\lim_{n \to \infty} {\rm Dim}^+(G_n)$ which is $f_d \cdot (1,2, \cdots, d+1) /|f_d|_1$, 
where $f_d$ is the unique Perron-Frobenius probability eigenvector of $A_d$. 
\end{thm}

\paragraph{}
To compare with inductive dimension ${\rm dim}^+$: if the maximal dimension of $G$ is $d$, then 
$c_d = \lim_{n \to \infty} {\rm dim}^+(G_n) = d+1$ 
because the largest dimensional simplices
replicate faster than lower dimensional ones. On the other hand, for the average
simplex dimension ${\rm Dim}^+$, the constant $C_d$ 
is always in the interval $[(d+1)/2,(d+1))$. If $d$ is positive then 
$$  C_d \in ((d+1)/2,d+1)  \; . $$

\begin{figure}[!htpb]
\scalebox{0.42}{\includegraphics{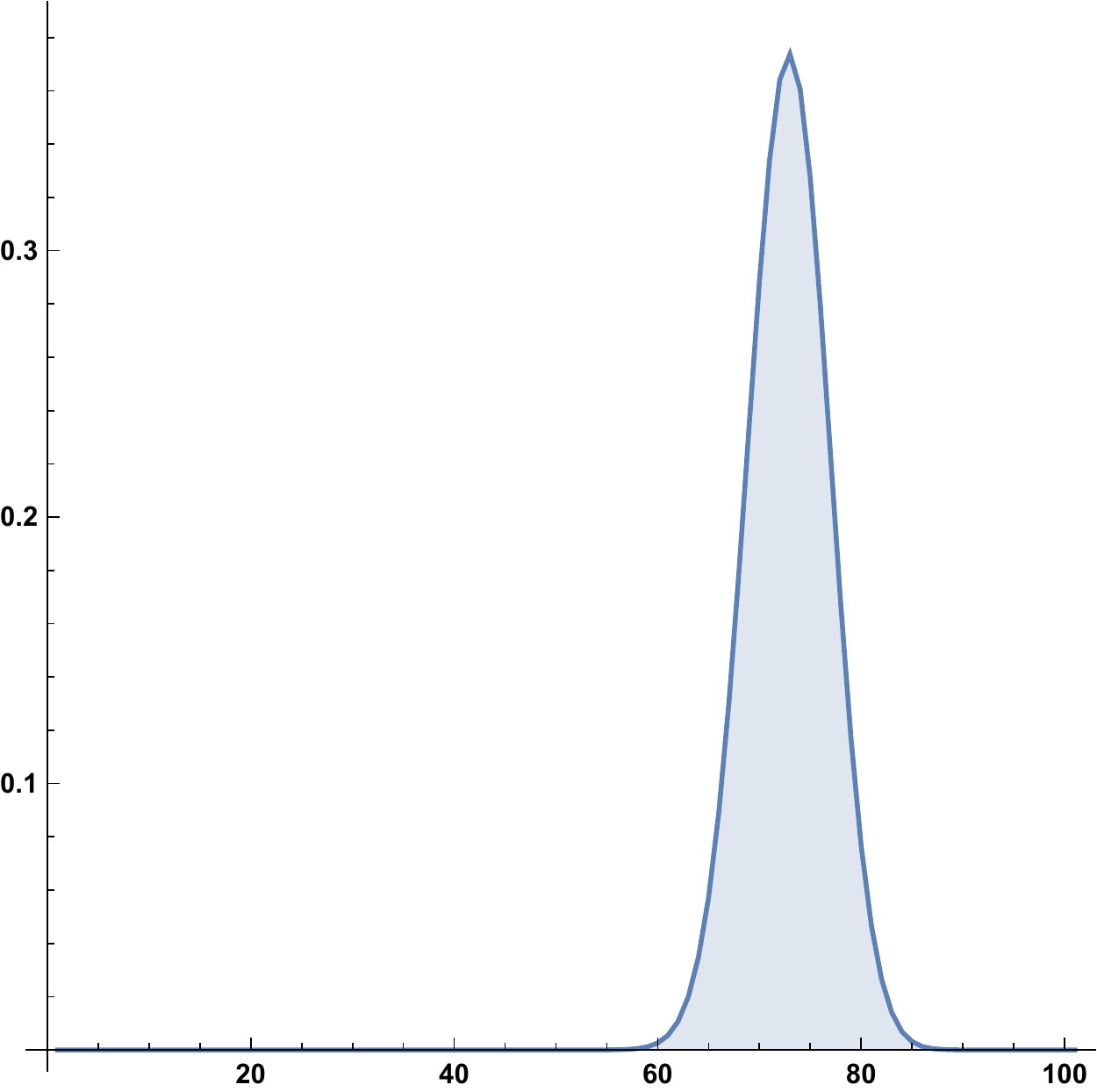}}
\scalebox{0.42}{\includegraphics{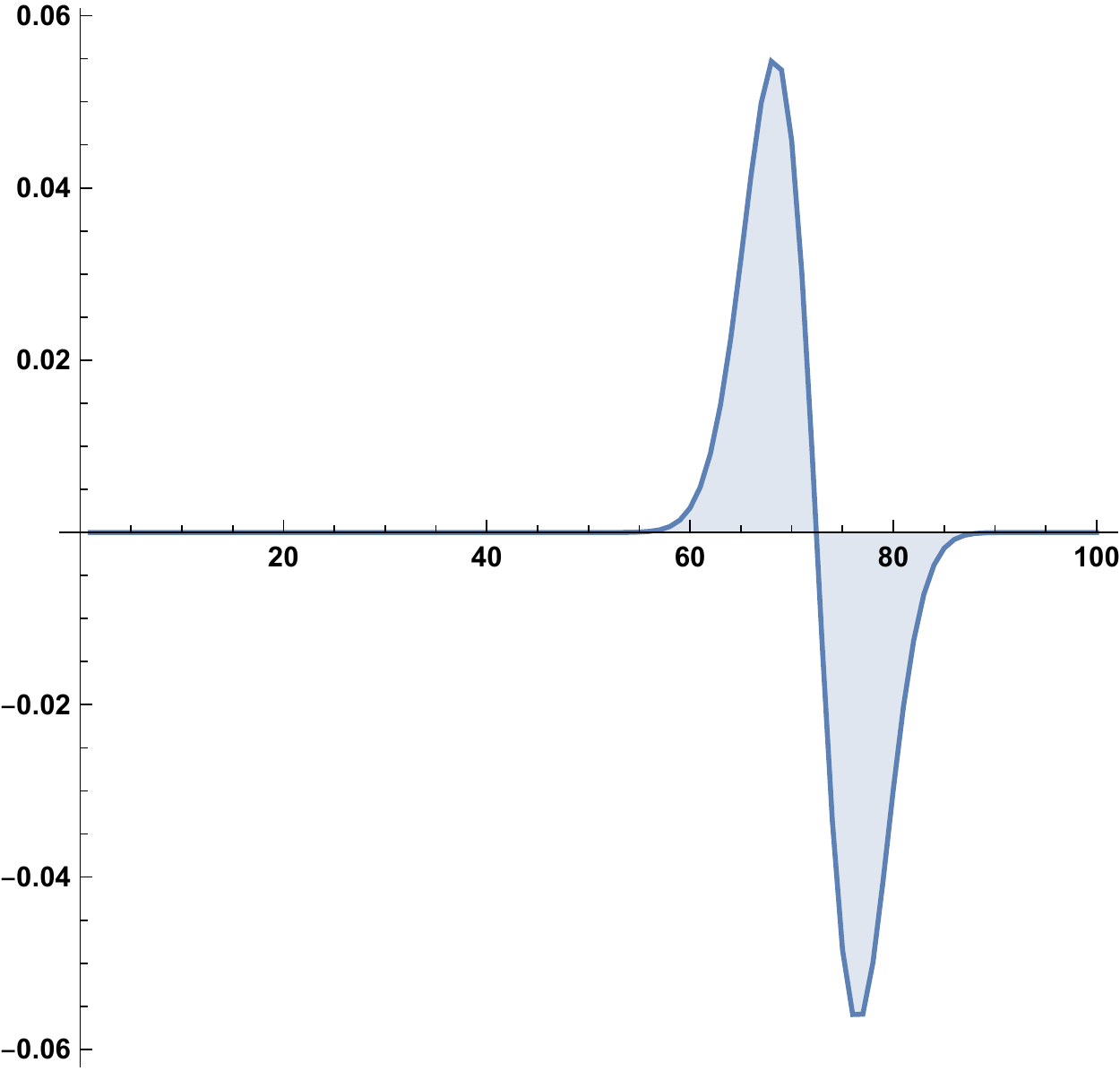}}
\label{eigenvector}
\caption{
The Perron-Frobenius eigenvector $f$ of $A_d$ in the case $d=100$. This vector defines  the limiting
distribution of the simplex cardinality for the limiting Barycentric refinement in 
dimension $d=100$. In the Barycentric limit, most simplices have about dimension $75$. To the right we
show the discrete derivative  $f'(k)=f(k+1)-f(k)$. The mean of $f$ is $72.828 ...$. This is
the expectation ${\rm Dim}^+(G_n)$ in the limit $n \to \infty$. The actual limit is
in dimension $d=100$ a rational number $p/q$, with $4423$ digit integers $p,q$. 
}
\end{figure}

\section{Remarks}

\paragraph{}
Various notions of dimension exist for discrete geometries. For finite abstract simplicial complexes
one has the {\bf maximal dimension} ${\rm max}(G)={\rm max}_{x \in G} |x|-1$ which counts the 
dimension of the largest simplex $x \in G$. The {\bf inductive dimension} \cite{elemente11,randomgraph} 
motivated from the corresponding notion for topological spaces \cite{HurewiczWallman} starts with the assumption
that the empty complex has dimension $-1$ and then defines the dimension of of a set of non-empty sets 
as $1$ plus the average dimension of the unit spheres $S(x)$ in $G$, where $S(x)$ is the 
pre-simplicial complex given by the set of sets in $G$ which either contain $x$ or are contained in $x$. 

\paragraph{}
For graphs $\Gamma=(V,E)$, the inductive dimension is one plus the average dimension of unit spheres $S(v)$
where the average is taken over all $v \in V$. The inductive dimension of the Whitney complex of $\Gamma$,
(the vertex sets of complete subgraphs of $G$). The inductive dimension of that complex $G$ 
is larger or equal than the dimension of $\Gamma$. The Barycentric refinement of a complex, the 
order complex of $G$ is always the Whitney complex of a graph, but there are complexes like 
$C_3=\{ \{ 1,2\},\{2,3\},\{3,1\} \}$ which are not Whitney complexes. 

\begin{figure}[!htpb]
\scalebox{0.32}{\includegraphics{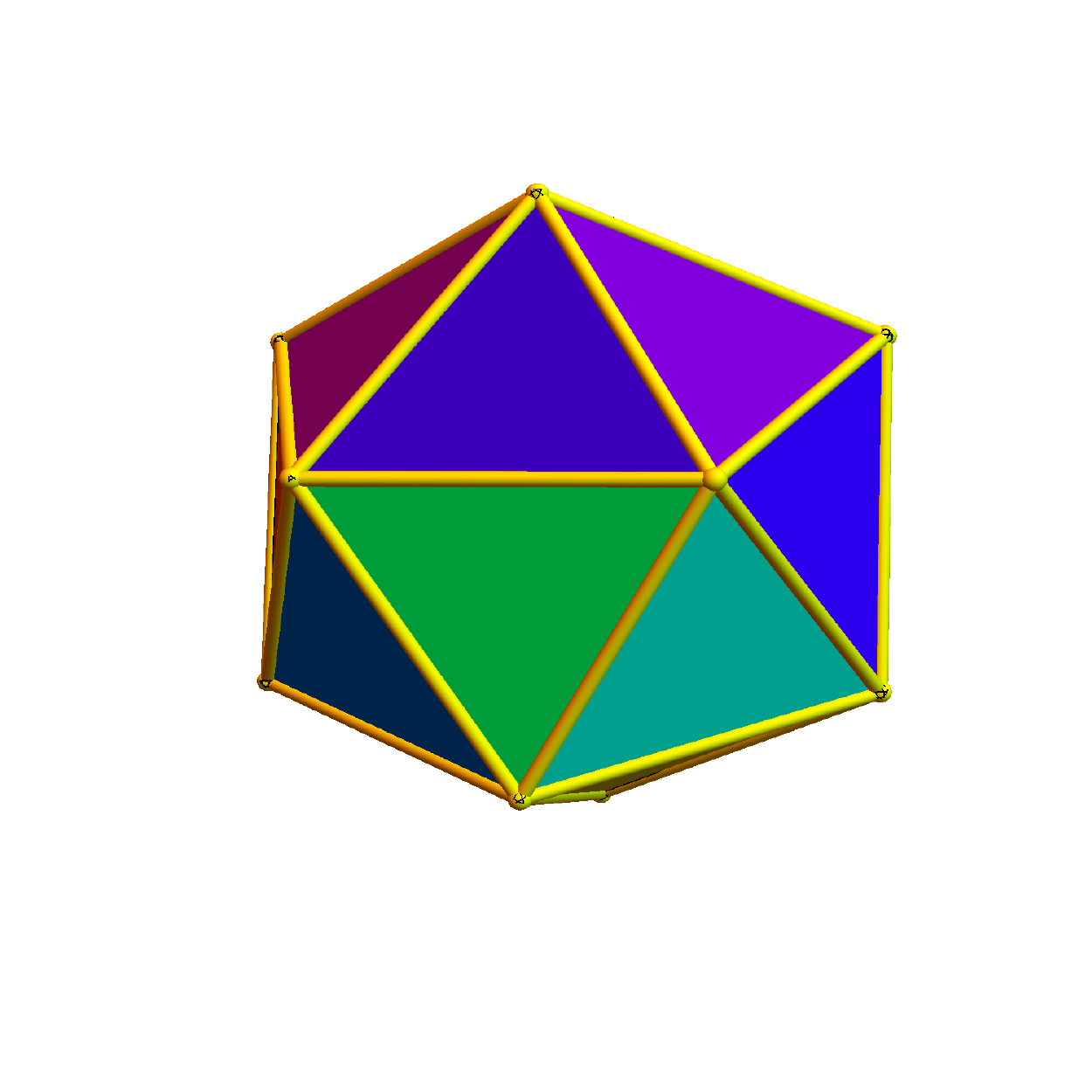}}
\scalebox{0.32}{\includegraphics{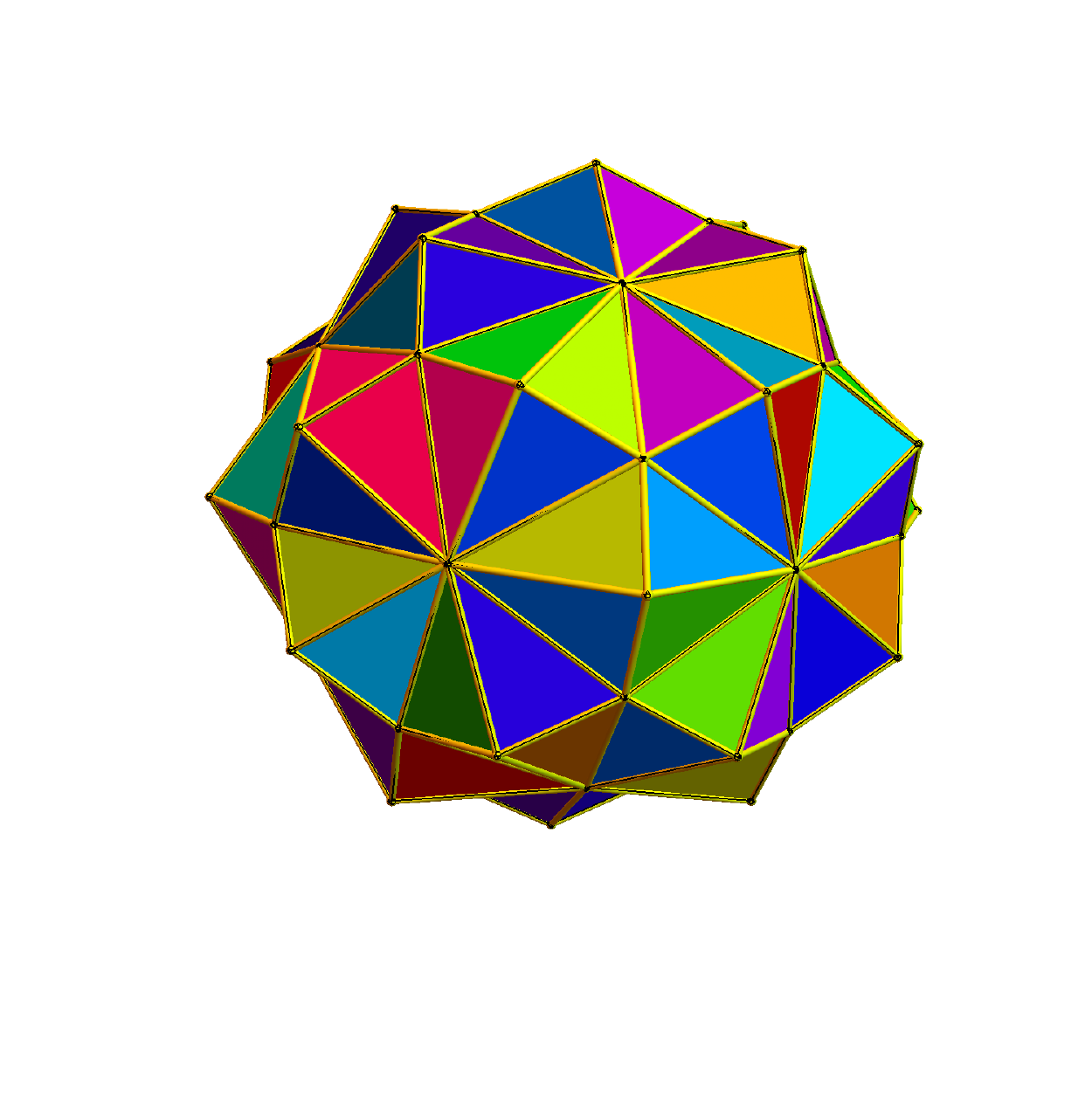}}
\scalebox{0.32}{\includegraphics{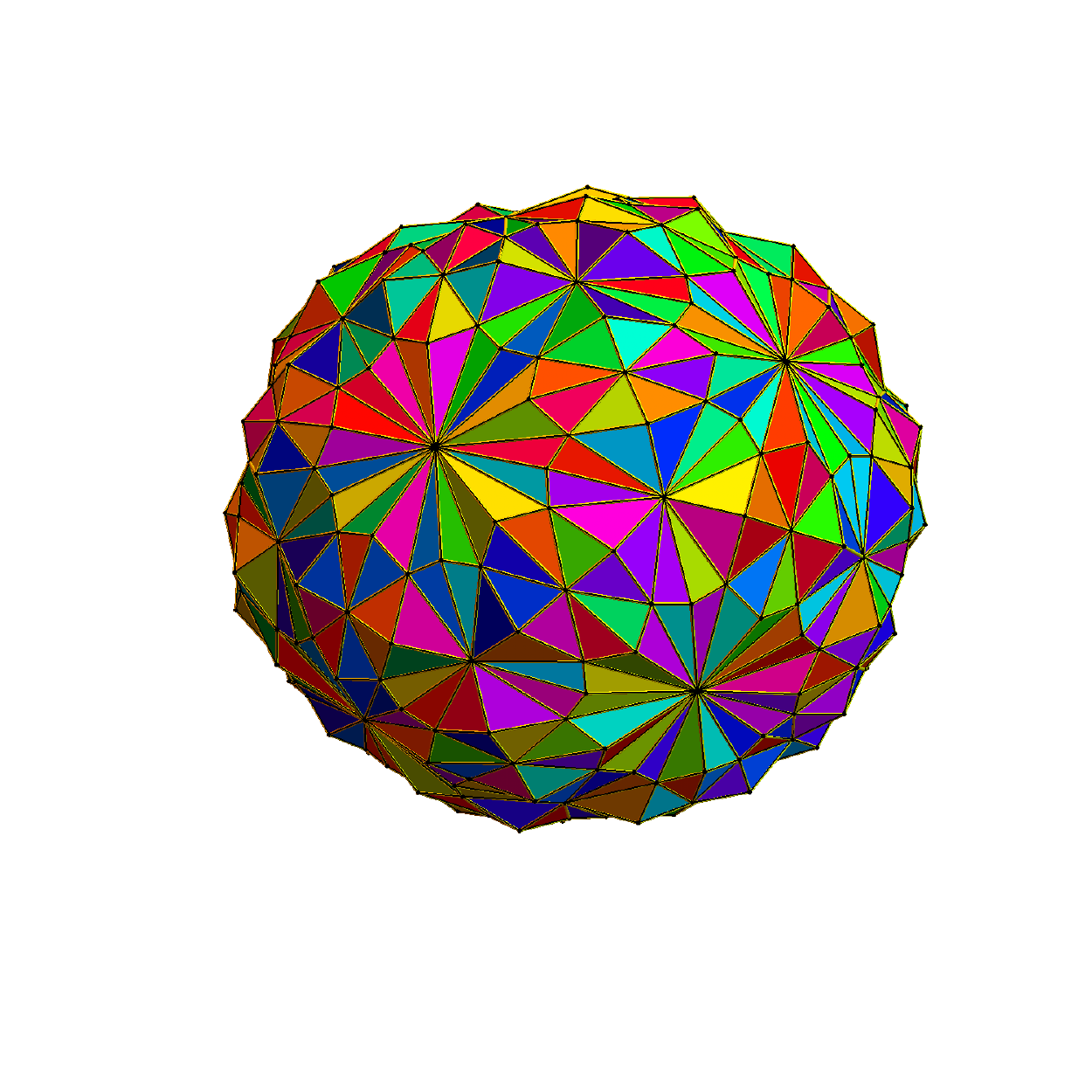}}
\label{icosahedron}
\caption{
The icosahedron complex $G$ with $f_G(t) = 1+12t+30t^2+20 t^2$ and its first two refinements $G_1,G_2$.
The inductive dimension of $G_k$ is ${\rm dim}(G_k) = 2$ so that ${\rm dim}^+(G_k)/2=3/2$. The average simplex cardinality
are ${\rm Dim}^+(G_0)=(1*12+2*30+3*20)/(1+12+30+20)=44/21=2.09..$, 
${\rm Dim}^+(G_1)=782/363=2.15..$ and ${\rm Dim}^+(G_2)=4682/2163=2.16..$.
The limiting dimension $\lim_{n \to \infty} {\rm Dim}^+(G_n)$ is $(1,3,2) \cdot (1,2,3)/(1+3+2)=13/6=2.1666..$. 
}
\end{figure}

\paragraph{}
Every simplicial complex $G$ also defines
a graph, in which $V=G$ are the vertices and two simplices are connected by an edge 
if one is contained in the other. The {\bf embedding dimension} of a graph \cite{EHT} 
is the minimal $d$ such that $G$ can embedded in $\mathbb{R}^d$ so that every edge length is $1$. 
The {\bf metric dimension} \cite{HararyMelter} 
is the minimal cardinality of a subset $C$ such that all other vertices are uniquely determined by the 
distances to vertices in $C$. Also motivated by corresponding notions put forward by Lebesgue and \v{C}ech
are notions of {\bf covering dimension} which looks at the minimal order of an open cover, the maximal
cardinality of open sets which a point can contain. In the discrete, this depends on the type of coverings
chosen as open sets. 

\paragraph{}
Both the maximal dimension as well as the inductive dimension are intuitive as they agree with classical 
topological dimensions in the case of discrete manifolds or varieties and both have arithmetic compatibility 
with the join operation. The inductive dimension captures subtleties of general networks as it allows for 
fractional dimension like {\bf Hausdorff dimension} and also shares an inequality 
${\rm dim}(X \times Y) \geq {\rm dim}(X) + {\rm dim}(Y)$ \cite{KnillKuenneth} 
which parallels the same inequality for Borel sets and Hausdorff dimension in the continuum 
\cite{Falconer}. (Of course, the notions in the discrete and the continuum 
are completely different but the analogy is remarkable). 

\paragraph{}
We study here a dimension ${\rm Dim}$ which is of a more statistical nature as it counts the average dimension
of the building blocks of a simplicial complex $G$. Unlike the inductive ${\rm dim}$ or maximal dimensions ${\rm max}$ 
which are rather ``plain" for discrete manifolds or varieties (or more generally for homogeneous =pure
simplicial complexes, the averaging dimension adds additional information for manifolds. 
It also shares the arithmetic compatibility property ${\rm Dim}(G \oplus H) = ({\rm Dim}^+(G) + {\rm Dim}(H)-1$ 
with the same property for maximal dimension ${\rm max}(G \oplus H) = {\rm maxdim}(G) + {\rm maxdim}(H)-1$ 
and for inductive dimension ${\rm dim}(G \oplus H) = {\rm dim}(G) + {\rm dim}(H)-1$. It is this $-1$ discrepancy which
makes the use of the augmented versions ${\rm Dim}^+,{\rm dim}^+,{\rm max}$ more attractive. 

\section{Questions}

\paragraph{}
{\bf (A)} The answer to the question
\begin{conj} Is ${\rm Dim}^+(G_1) \geq {\rm Dim}^+(G)$ always true? \end{conj}
has been ``yes" in all cases seen so far. It is a linear algebra problem: 
we have to show that f-vectors satisfying the {\bf Kruskal-Katona} conditions have the property
$$   \frac{(Af \cdot v)}{(1+|Af|_1)} \geq \frac{(f \cdot v)}{(1+|f|_1)} \; ,  $$
where $v$ is the Perron-Frobenius eigenvector of the {\bf Barycentric refinement matrix}
$$  A_{ij} = i! {\rm Stirling}(j,i)  \; . $$
The inequality for vectors $f$ does not hold for all vectors $f$ but it appears to be
true however for f-vectors which appear for simplicial complexes. 

\paragraph{}
{\bf Example:} for the first complete graphs $G=K_n$, we get 
${\rm Dim}^+(G_1)-{\rm Dim}^+(G) = \delta_k$ with 
$$ \delta = (0,1/6,5/13,91/150,448/541, ...)  $$

\paragraph{}
{\bf (B)} One can ask whether
\begin{conj} $\mathcal{G}_{p/q} = \{ G \; | \; f'_G(1)/f_G(1) =p/q\}$  is finite on the class of varieties $G$ \end{conj} 
($d$-varieties as defined recursively as complexes for which all unit spheres $S(x)$ are $(d-1)$-varieties with the 
induction assumption that the empty complex is a $-1$ dimensional variety. 

\paragraph{}
The set $\mathcal{G}_{p/q} = \{ G \; | \; {\rm Dim}^+(G) =p/q\}$ is infinite for arbitrary large 
one dimensional complexes with $v_0=n,v_1=m$ for example satisfying 
$n+2m/(n+m+1)=4/3$ by picking a number $m$ of edges, then choosing $n=2m-4$. But there are only 
finitely many examples which are varieties. By adding isospectral
examples, one get arbitrary large, arbitrary high dimensional complexes with the same ${\rm Dim}^+(G) = c$.  \\

\paragraph{}
{\bf (C)} The following question could have a simple answer, we just don't know yet. It is a pure
linear algebra question.  The Perron-Frobenius eigenvector $v_d=(v_0, \dots, v_d)$ of the matrix $A_d$
defines a probability distribution $f_d$ on $[0,1]$ by defining the real function $f_d(x)=v_d([d*x])$ 
normalized so that it is a probability density function. The sequence $f_d$ converges in the weak-* sense
as measures to a Dirac measures $1_{x_0}$ with $x_0 \sim 0.72...$. 

\begin{conj} Can we locate $x_0$?  \end{conj} 
Figure~(\ref{eigenvector}) shows the case $d=100$, where we see the Perron-Frobenius eigenvector
to the eigenvalue $\lambda=100!$.  \\

\paragraph{}
{\bf (D)} The average inductive dimension 
${\rm E}_p[{\rm dim}]$ on the Erd\"os-Renyi space $G(n,p)$ satisfies
$$ d_{n+1}(p) = 1+\sum_{k=0}^n \B{n}{k} p^k (1-p)^{n-k} d_k(p) \; , $$
where $d_0=-1$. Each $d_n$ is a polynomial in $p$ of degree $\B{n}{2}$.  We have for example:
$$ d_1(p) = 0, d_2(p)= p, d_3(p)= p(2-p + p^2), d_4(p) = p(3 - 3p + 4p^2 - p^3 - p^4 + p^5) \; .  $$

\begin{conj}
What is the average simplex cardinality on $G(n,p)$?
\end{conj}

\paragraph{}
By counting through all graphs, we got the following values
for $p=1/2$. 
For $n=1$ it is $1/2$, for $n=2$ it is $5/6$ for $n=3$, 
it is $35/32$ for $n=4$ it is $6593/5040$ for $n=5$ it is $18890551/12673024$.

\paragraph{}
We would have to be able to compute ${\rm E_p}[ \log(f(t)) ]$ as a function of $p$ and $t$
and then differentiate at $t=1$ to get the average simplex cardinality. While we know the
expectation of the simplex generating function $f(t)$ itself, we do not have formulas for 
moments ${\rm E}_p[ f(t)^k]$. 

\paragraph{}
{\bf (D)} Among all graphs with $n$ vertices, which one maximizes
$$  \delta(G) = {\rm Dim}^+(G) -{\rm dim}^+(G)/2  \; ? $$
We know it is the cyclic graph $C_4 = K_{2,2} = 2 \oplus 2$ in the case $n=4$ where
it the difference is $1/3$ and the 
bipartite graph $K_{3,3} = 3 \oplus 3$ (utility graph) in the case $n=6$,
where the difference is $1/2$. As $f_{K(n,n)}(t) = 1+ 2n t + n^2 t^2$ we 
have ${\rm Dim}^+(K_{n,n}) = (2n+2n^2)/(1+2n+n^2) = 2n/(1+n)$ and 
$\delta(K_{n,n}) = (n-1)/(n+1)$. The bipartite graphs $K_{n,n}$ can not be maximal
in general as the maximum goes to infinity for $n \to \infty$ and 
$\delta(K_{n,n})<1$. 

\begin{conj}
For which graphs with $n$ vertices is $\delta(G)$ maximal?
\end{conj}

\paragraph{}
{\bf (E)} Having looked at the average cardinality $m(G)={\rm Dim}^+(G)$ , 
one can wonder about the {\bf variance} ${\rm Var}^+(G) = \sum_{x \in G} (|x|-m(G))^2/|G^+|$ 
and {\bf higher moments} $\sum_{x \in G} (|x|-m(G))^k/|G^+|$ and their 
behavior in the Barycentric limit or its typical value on Erd\"os-R\'enyi
spaces. 

\begin{conj}
Are there a limiting laws for the moments of the simplex cardinality? 
\end{conj}

\paragraph{}
For complete complexes $K_n$, we computed the variance $v_n = {\rm Var}^+(K_n)$
and got the values $v_1=1/8,v_2=1/4,v_3=15/32,v_4=3/4,v_5=135/128$. 
For example, for $G=K_1=\{ \{1\} \}$, we have ${\rm Dim}^+(G)=1/2$ and
${\rm Var}^+(G) = (1-1/2)^2/2 = 1/8$. The sequence $v_n$ grows
asymptotically linearly as the variance of
the simplex cardinality grows exponentially under Barycentric refinements. 
It must do so as there is a limiting distribution given by the Perron-Frobenius
eigenfunction of the Barycentric refinement operator shown in Figure~(\ref{eigenvector})
in the case $d=100$. 

\section{Illustrations}

\begin{figure}[!htpb]
\scalebox{0.82}{\includegraphics{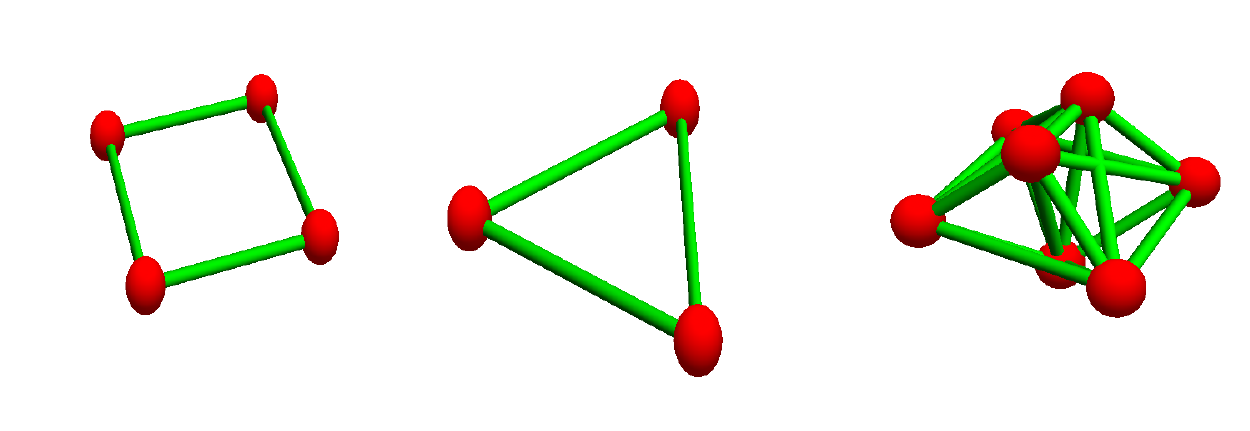}}
\label{example1}
\caption{
The cyclic graph $C_4$ has average simplex cardinality $(4+4*2)/9=4/3$. 
The complete graph $K_3$ has average simplex cardinality $(3+3*2+1*3)/8=3/2$.
The join $C_4 \oplus K_3$ is 4-dimensional with f-vector $(7, 19, 25, 16, 4)$ 
and average simplex cardinality $(7+19*2+25*3+16*4+4*5)/(1+7+19+25+16+4) =17/6$.
The example illustrates the compatibility with join: $4/3+3/2=17/6$. 
}
\end{figure}

\begin{figure}[!htpb]
\scalebox{0.52}{\includegraphics{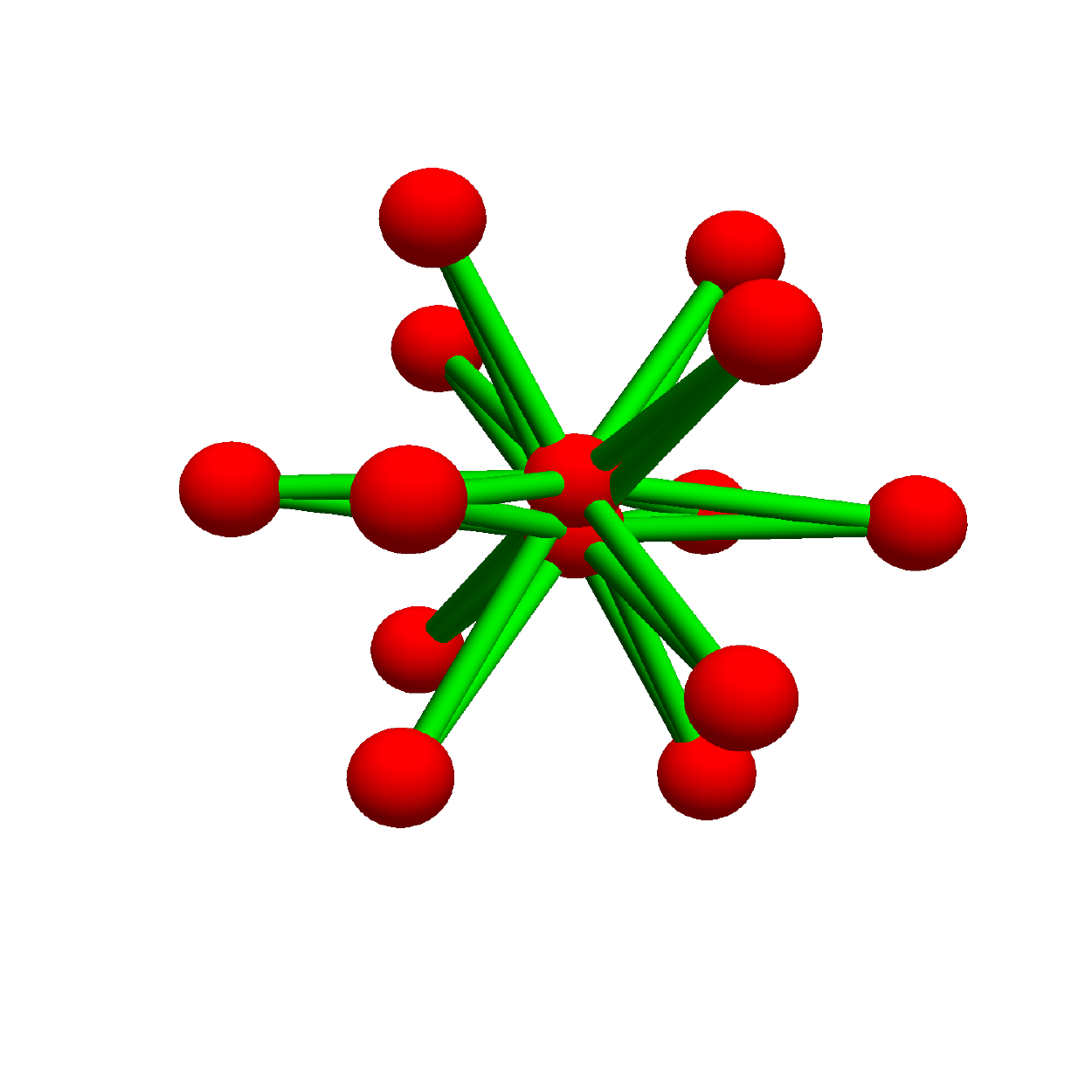}}
\label{example2}
\caption{
The join of two point graphs $12$ and $2$ gives the bipartite
graph $K_{12,2}$. We have ${\rm dim}^+(12)={\rm dim}^+(2)=1$ and 
${\rm dim}^+(K_{12,2})=2$. Also ${\rm Dim}^+(12)=12/13, {\rm Dim}^+(2)=2/3$ and
${\rm Dim}^+(K_{12,2})=62/39$. This illustrates the additivity of ${\rm dim}^+$ 
and ${\rm Dim}^+$. 
}
\end{figure}

\begin{figure}[!htpb]
\scalebox{0.35}{\includegraphics{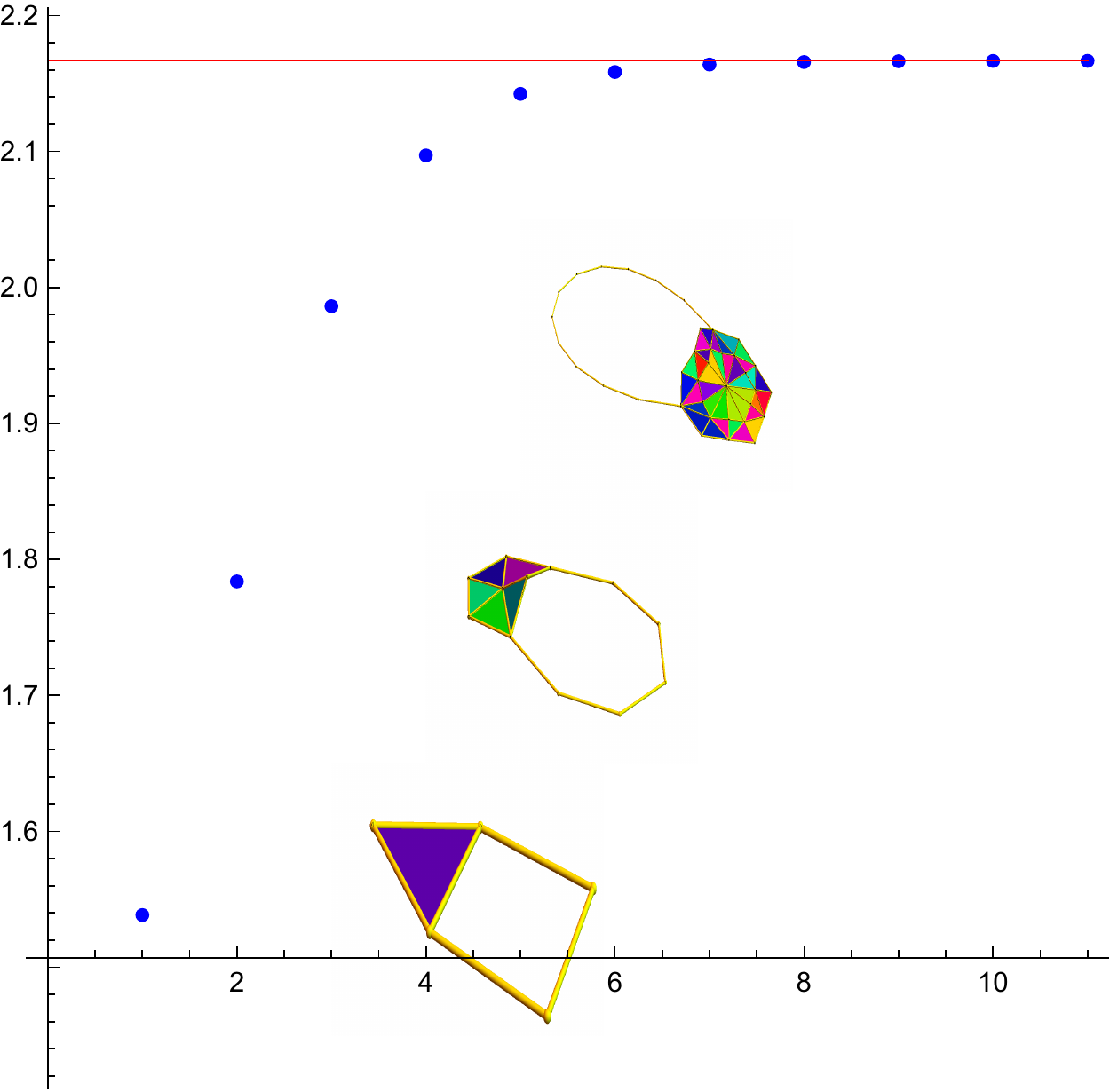}}
\scalebox{0.35}{\includegraphics{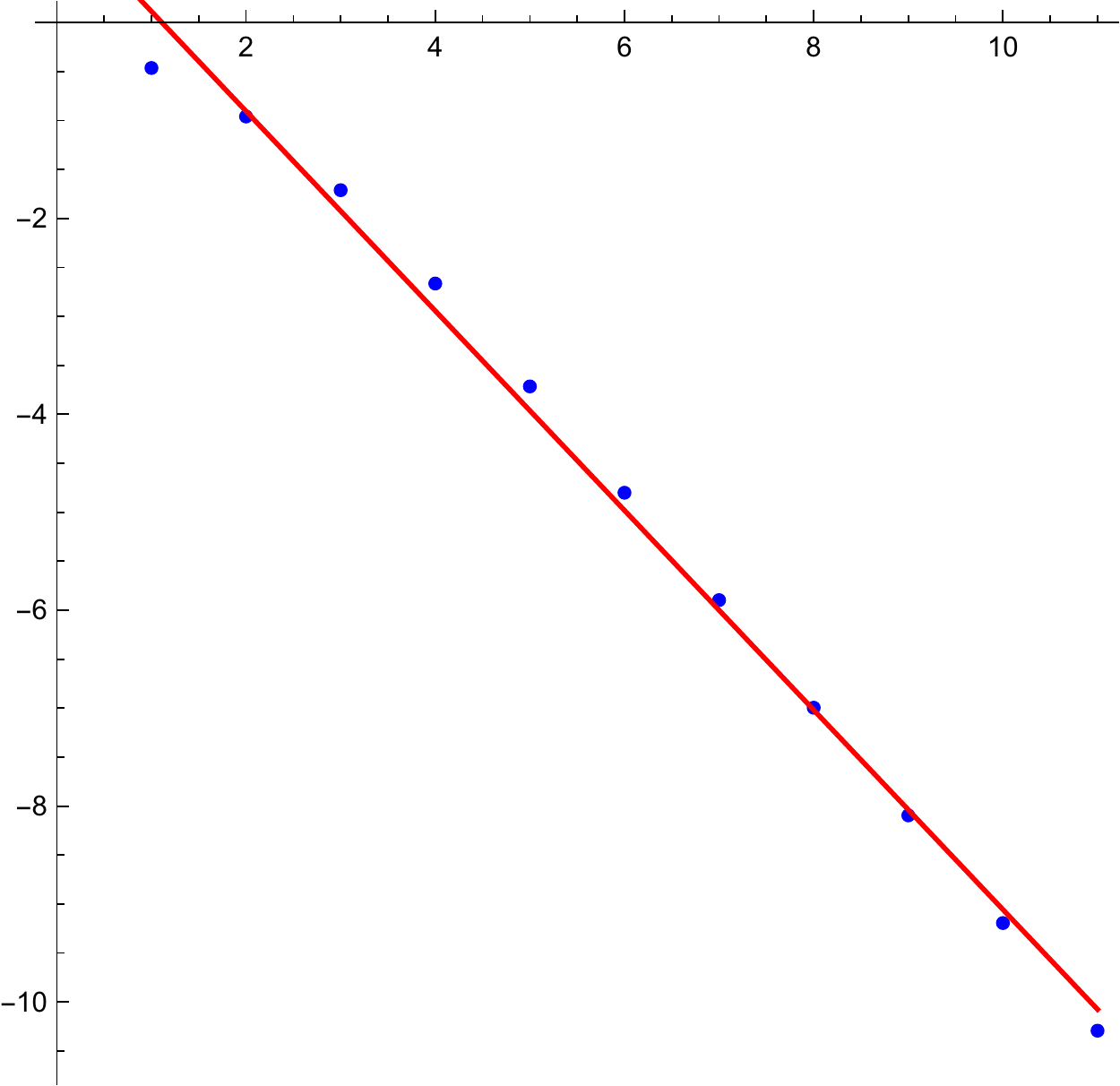}}
\label{housegraph}
\caption{
We see the average simplex density for the house graph and 
the first Barycentric refinements. The upper bound is $c=13/6 = (1,3,2) \cdot (1,2,3)/(1+3+2)$
determined by the Perron-Frobenius eigenvector $(1,3,2)$ of the Barycentric refinement operator $A$.
The right picture shows $\log( {\rm Dim}^+(G_n) - c)$. It is a monotone sequence if 
${\rm Dim}^+(G_1)  \geq {\rm Dim}^+(G)$ holds. Verifying this inequality would be 
a linear algebra problem. Does the Kruskal-Katona constraint for $f$-vectors imply that 
$(A f) \cdot (1, \cdots, d)/((A f) \cdot (1,\dots,1))$
is larger or equal than $f \cdot (1, \cdots, d)/(f \cdot (1,\cdots,1))$?
}
\end{figure}

\begin{figure}[!htpb]
\scalebox{0.42}{\includegraphics{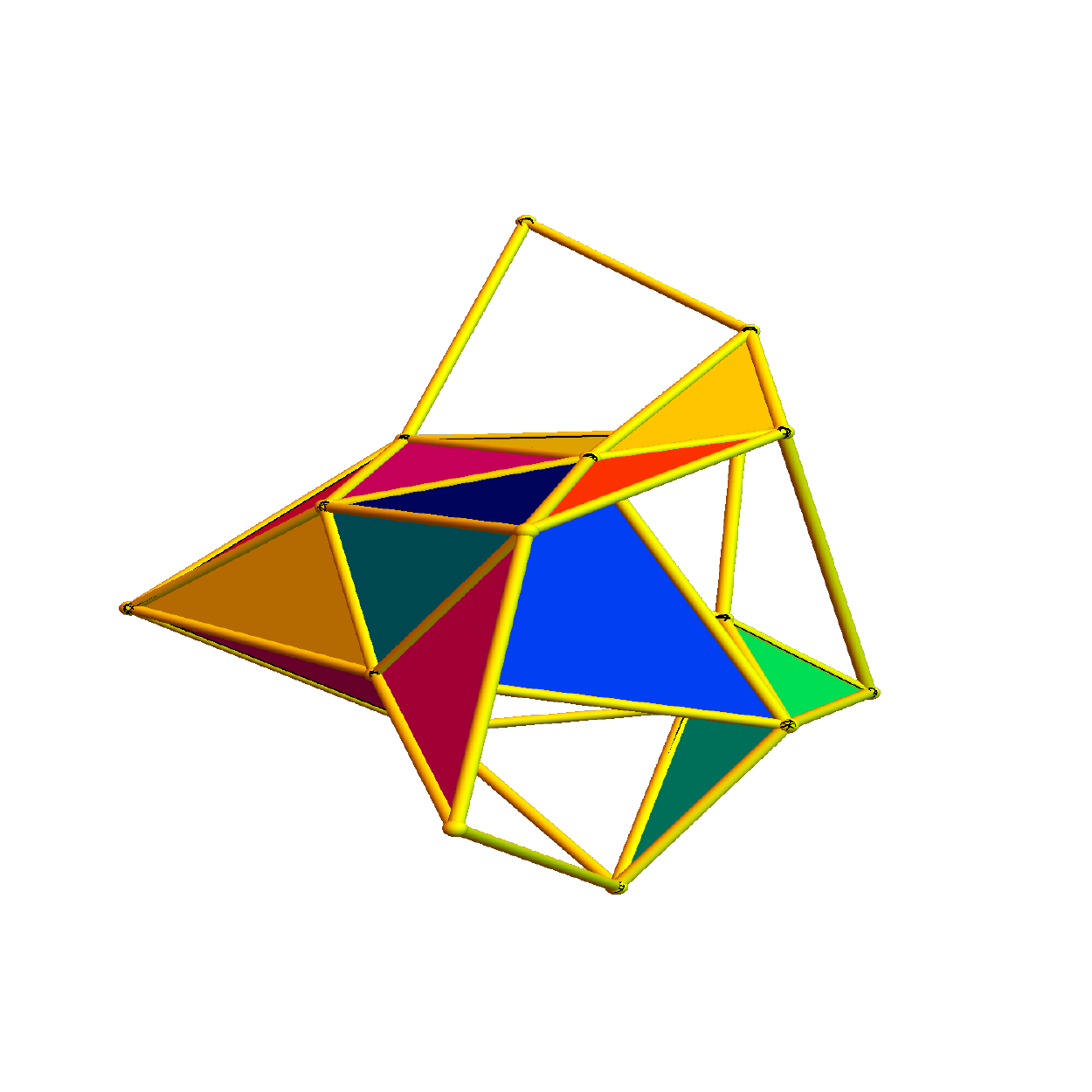}}
\label{example}
\caption{
A graph with f-vector $(15, 36, 16, 1)$. Its inductive dimension is
$8587/4725= 1.817...$. Its average simplex cardinality is $139/69=2.01$.
The augmented volume is $f(1)=1+15+36+16+1=69$. The derivative $f'$ of the
simplex generating function is $f'(t)=15 + 72t + 48t^2 + 4t^3$ which 
satisfies $f'(1)=139$. The inequality $1.4087... = (1+{\rm dim}(G))/2 \leq {\rm Dim}^+(G) = 2.01..$.
is satisfied. The Barycentric refinement $G_1$ satisfies 
${\rm dim}(G_1)=131002727/65345280=2.0048...$. 
It is an old theorem which assures ${\rm dim}(G_1) \geq {\rm dim}(G)$. 
We have ${\rm Dim}^+(G_1) = 84/37=2.270...$ which is larger than 
${\rm Dim}^+(G) = 139/69=2.014...$. It is one of the open problems to prove this. 
}
\end{figure}

\begin{figure}[!htpb]
\scalebox{0.52}{\includegraphics{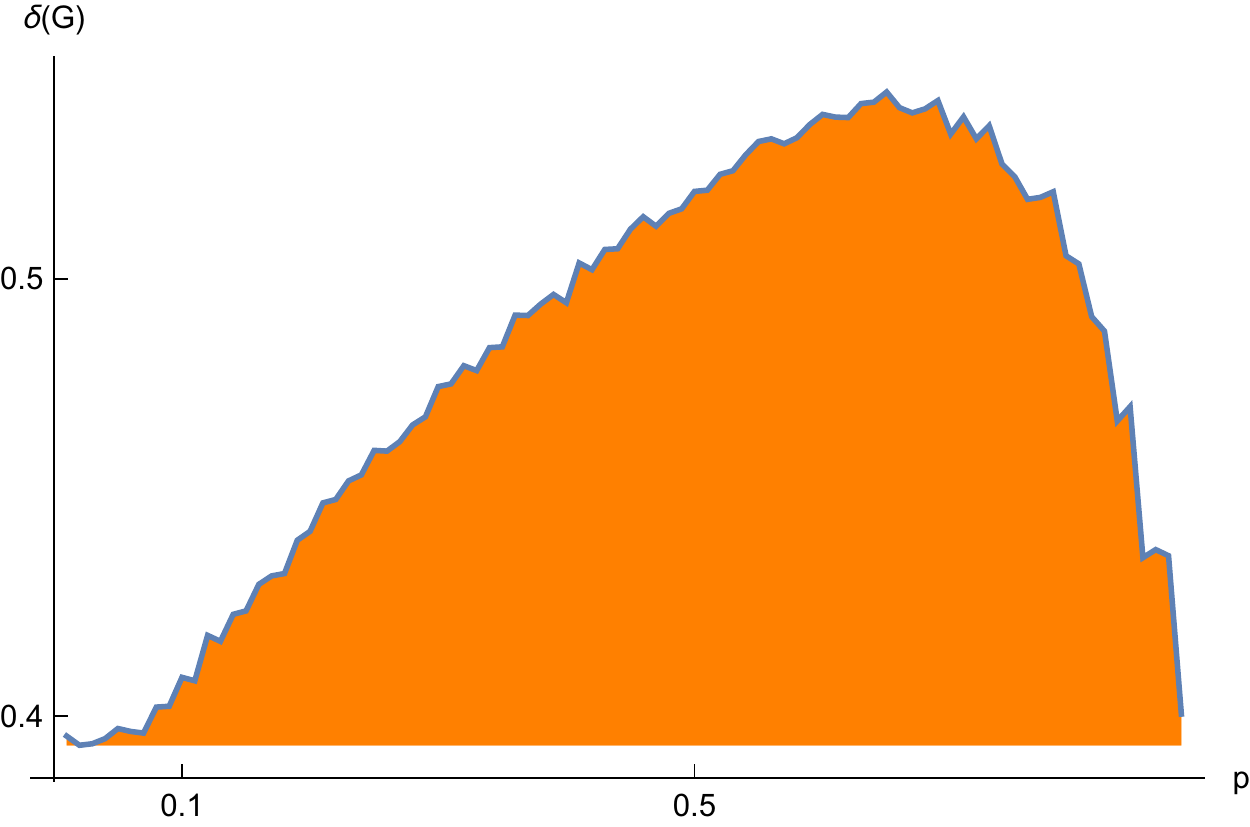}}
\label{dimensiondifference}
\caption{
We see $\delta(G) = {\rm Dim}^+(G) -{\rm dim}^+(G)/2$ which according to the theorem
is positive. Every point is an average over 1000 Erd\"os-Renyi graphs \cite{erdoesrenyi59} 
with 10 vertices
and edge probability $p$. The difference $\delta(G)$ is quite significant 
but for $p=1$, where we have the complete graph, it drops down to $0$. 
For $p=0$, where ${\rm Dim}^+(G)=n/(n+1)$ and ${\rm dim}^+(G)=1$ the difference is
$n/(1+n)-1/2$ which is in the case $n=10$ equal to $9/22$. 
}
\end{figure}

\begin{figure}[!htpb]
\scalebox{0.42}{\includegraphics{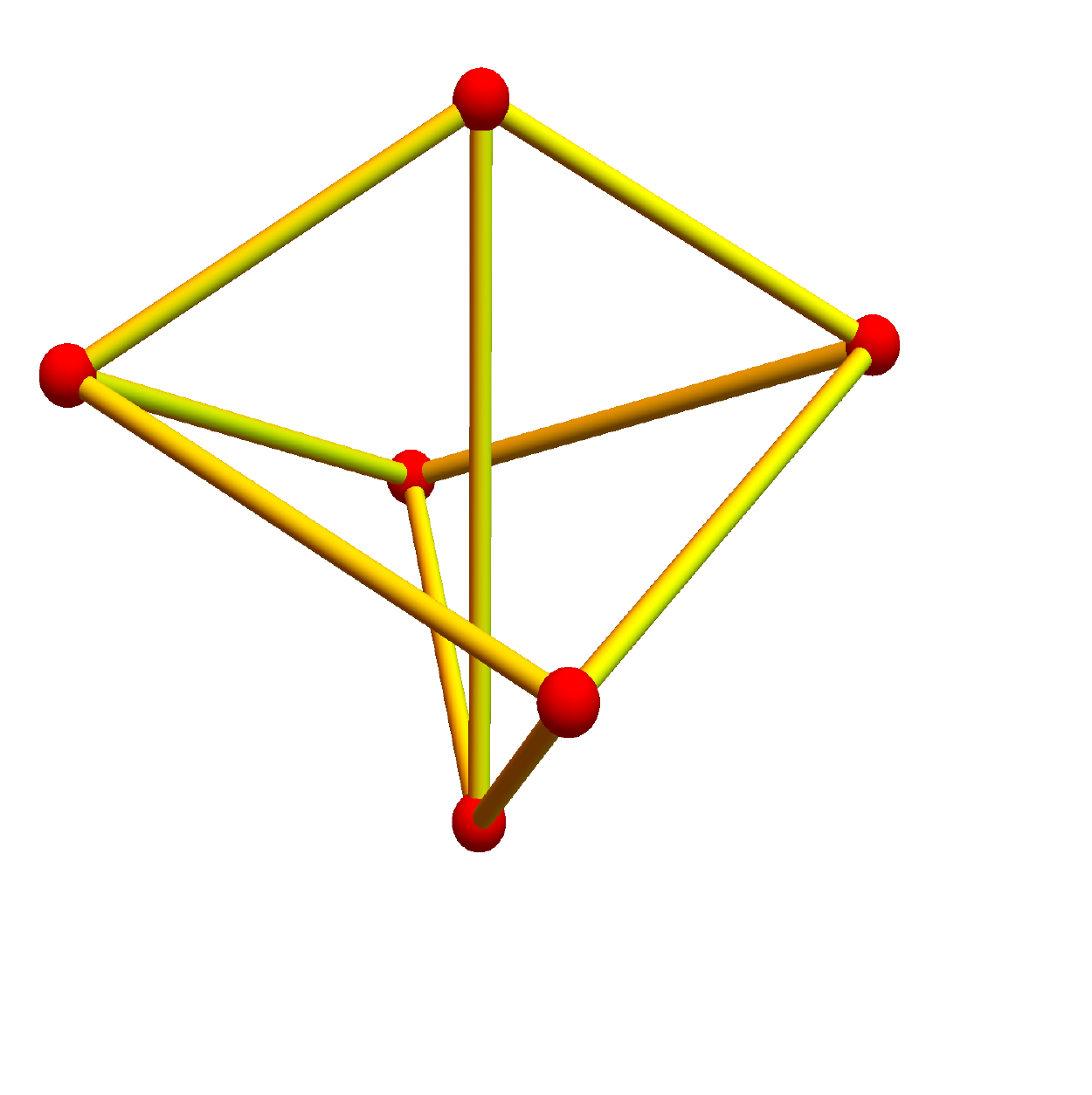}}
\scalebox{0.42}{\includegraphics{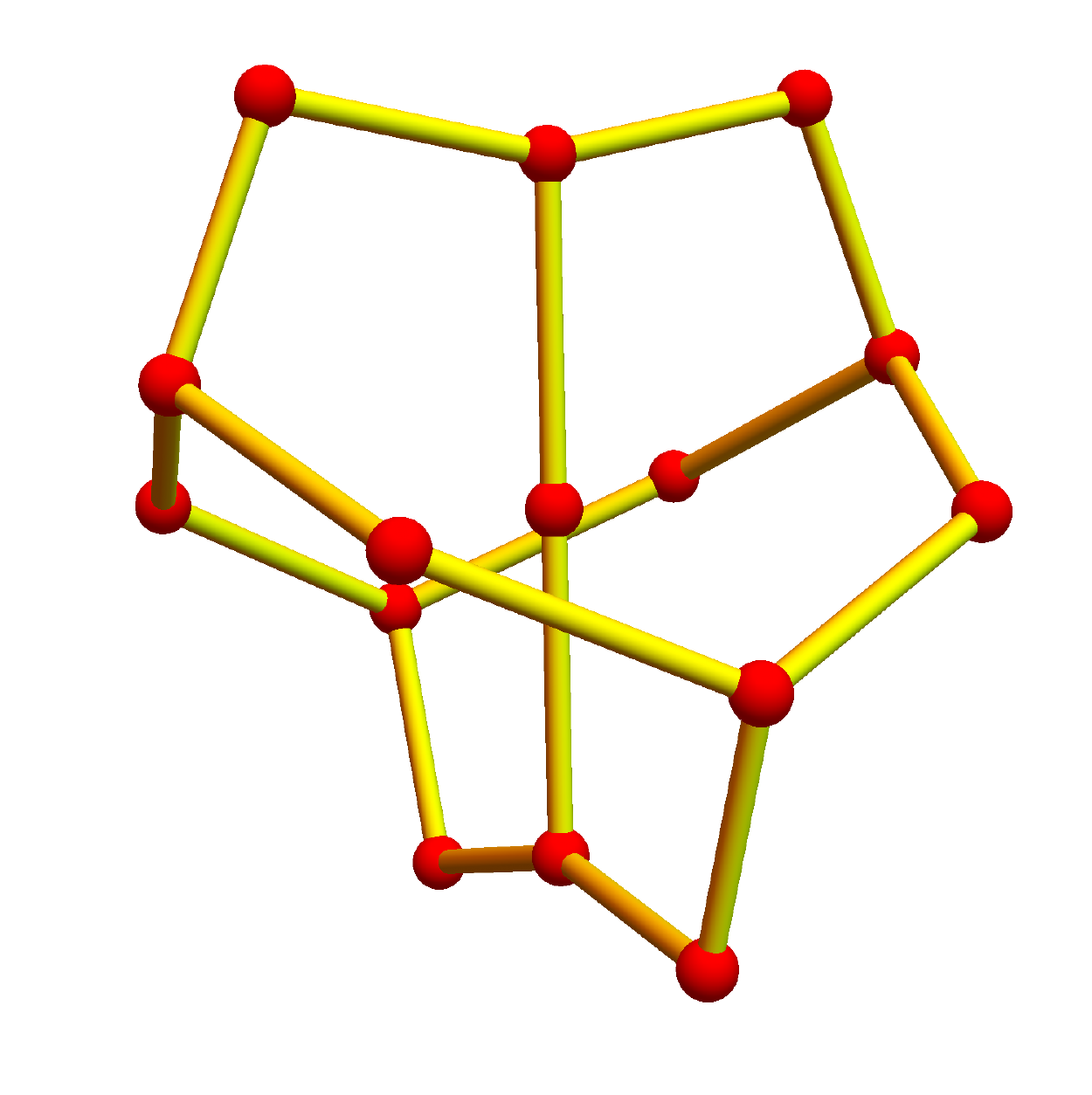}}
\label{maximaldimension}
\caption{
Among all connected graphs with $6$ vertices, the {\bf utility graph} $K_{3,3}$ is the 
one with maximal $\delta(G) = {\rm Dim}^+(G) -{\rm dim}^+(G)/2$. 
It is a $3$-regular graph of diameter $2$ which can not be embedded  of the complete
bipartite graph $K_{3,3} = 3 \oplus 3$ in $\mathbb{R}^2$. 
The picture shows a natural embedding in $\mathbb{R}^3$. In the case $n=2$, the extreme
case was $C_4 = 2 \oplus 2$. To the right, we see the Barycentric refinement, which
has the same $\delta(G)$. The case $\delta(G_1)=\delta(G)$ can not hold any more
in higher dimensions.
}
\end{figure}

\paragraph{}
The figures were generated in Mathematica. We give some code in the next section. 
As usual, the code can be copy-pasted from the ArXiv posting of this paper. 
The Wolfram language also serves as good pseudo code
which could with ease be translated into any other programming language. 

\section{Code}

\paragraph{}
The following Mathematica code implements the computation of the 
augmented inductive dimension and the augmented average simplex dimension
and computes it for some random simplicial complex. 
It then compares the two numbers 
${\rm dim}^+(G)/2$ and ${\rm Dim}^+(G)$ which appear in the inequality. 

\begin{small} 
\lstset{language=Mathematica} \lstset{frameround=fttt}
\begin{lstlisting}[frame=single]
Generate[A_]:=Delete[Union[Sort[Flatten[Map[Subsets,A],1]]],1]
R[n_,m_]:=Module[{A={},X=Range[n],k},Do[k:=1+Random[Integer,n-1];
  A=Append[A,Union[RandomChoice[X,k]]],{m}];Generate[A]];
Fvector[G_]:=Delete[BinCounts[Map[Length,G]-1],1];
Vol[G_]:=Total[Fvector[G]]+1;
f[G_]:=Module[{fv=Fvector[G]},1+Sum[ff[[k]]*t^k,{k,Length[ff]}]];
Dim[G_]:=Range[Length[Fvector[G]]].Fvector[G]/Vol[G];
Var[G_]:=Module[{m=Dim[G],v=Fvector[G],V=Vol[G]},
   v.Table[(k-m)^2,{k,Length[v]}]/V];
sphere[H_,x_]:=Module[{U={},n=Length[H]},Do[
  If[Sort[H[[k]]]!=Sort[x] && (SubsetQ[x,H[[k]]] 
      || SubsetQ[H[[k]],x]),U=Append[U,H[[k]]]],
  {k,Length[H]}]; Table[Complement[U[[k]],{x}],{k,Length[U]}]];
dim[H_,x_]:=Module[{S},If[Length[H]==0,S={},S=sphere[H,x]]; 
  If[Max[Map[Length,H]]<=1,0,If[S=={},0,dim[S]]]];
dim[H_]:=Module[{n=Length[H],U},
  If[n==0,U={},U=Table[sphere[H,H[[k]]],{k,n}]];
  If[Length[U]==0,-1,Sum[1+dim[U[[k]]],{k,n}]/n]];
JoinComplex[G_,H_]:=Module[{HH=H+Max[G]+1,K}, K=Union[G,HH];
  Do[Do[K=Append[K,Union[G[[k]],HH[[l]]]],
  {k,Length[G]}],{l,Length[HH]}];K];

Do[G=R[6,6];Print[N[{(dim[G]+1)/2,Dim[G]}]],{n,5}]
CompleteComplex[n_]:=Generate[{Range[n]}];
Do[G=CompleteComplex[n];Print[{(dim[G]+1)/2,Dim[G]}],{n,2,5}]
CycleComplex[n_]:=Table[{k,Mod[k,n]+1},{k,n}];
Do[G=CycleComplex[n];Print[{(dim[G]+1)/2,Dim[G]}],{n,3,6}]

House =Generate[{{2,3,5},{1,4},{1,2},{3,4}}]; 
Rabbit=Generate[{{1,2,3},{3,4},{3,5}}];
RabbitHouse=JoinComplex[Rabbit,House];
{Dim[House],Dim[Rabbit],Dim[RabbitHouse]} 
{dim[House]+1,dim[Rabbit]+1,dim[RabbitHouse]+1} (*Patience!*) 
\end{lstlisting}
\end{small}

\paragraph{}
Note that in the just done computation of the inductive dimensions of the house graph
and rabbit graphs, we compute directly the dimensions of the Barycentric
refinements (the Whitney complexes of the graphs) because we average over
the simplices. The just given
code gives for the average cardinalities ${\rm Dim}^+({\rm House}) = 20/13$,
${\rm Dim}^+({\rm Rabbit})= 3/2$ and ${\rm Dim}^+({\rm RabbitHouse})=79/26$ and for the 
inductive dimensions ${\rm dim}^+({\rm House})=61/24, {\rm dim}^+({\rm Rabbit})=13/5$, 
and ${\rm dim}^+({\rm RabbitHouse})=617/120$. 
In Figure~(\ref{betresalinger}), we have noted the dimensions of the graphs themselves.
The code for that is open source and given for example in \cite{KnillWolframDemo1}.

\paragraph{}
We compute now the constants $C_d$, then the limiting value $C_d/d = 0.72...$
which for $d=500$ is $0.722733$. The limiting value is a real number which 
we can not yet identify. Is $\lim_{d \to \infty} C_d/d$ rational or irrational 
for example? It it algebraic or transcendental.

\begin{small} 
\lstset{language=Mathematica} \lstset{frameround=fttt}
\begin{lstlisting}[frame=single]
A[n_]:=Table[StirlingS2[j,i]*i!,{i,n+1},{j,n+1}]; 
EV=Eigenvectors;
c[k_] :=Module[{v=First[EV[    A[k]]]},v.Range[k+1]/Total[v]];
cN[k_]:=Module[{v=First[EV[1.0*A[k]]]},v.Range[k+1]/Total[v]];

ListOfConstants      = Table[c[d],{d,0,10}]
LimitingValue        = cN[500]/500  
DimensionIntervals   = Table[N[{(d+1)/2,cN[d],d+1}],{d,0,10}]
\end{lstlisting}
\end{small}

\vfill

\bibliographystyle{plain}

\end{document}